%% file: signature_published.tex
\documentclass[12pt]{amsart}

\usepackage{amssymb}
\usepackage{scalefnt}
\usepackage{setspace}
\usepackage{graphicx}
\usepackage{color}
\usepackage{transparent}
\usepackage{float}
\usepackage{blkarray}
\usepackage{url}

\newtheorem{thm}{Theorem}
\newtheorem*{theorem}{Theorem}
\newtheorem{prop}[thm]{Proposition}

\newtheorem{lem}[thm]{Lemma}
\newtheorem{cor}[thm]{Corollary}

\theoremstyle{definition}

\newtheorem{example}[thm]{\sc Example}
\newtheorem{question}[thm]{\sc Question}

\theoremstyle{remark}
\newtheorem{remark}[thm]{\sc Remark}
\theoremstyle{conjecture}
\newtheorem{conj}[thm]{Conjecture}

\newtheorem*{conjecture}{Conjecture}
\newtheorem*{thmA}{Theorem A}
\newtheorem*{propA}{Proposition D}
\newtheorem*{lemA}{Lemma B}
\newtheorem*{remarkA}{\sc Remark C}

\newcommand{\R}{\mathbf{R}} 
\newcommand{\C}{\mathbf{C}} 
\newcommand{\Z}{\mathbf{Z}} 

\makeatletter
\let\@wraptoccontribs\wraptoccontribs
\makeatother

\begin{document}
\title{Signature, positive Hopf plumbing and the~Coxeter transformation}
\author{Livio Liechti}
\address{Livio Liechti, Mathematisches Institut, Universit\"at Bern, Sidlerstrasse 5, CH-3012 Bern, Switzerland}
\email{livio.liechti@math.unibe.ch}
\contrib[With appendix by]{Peter Feller}
\contrib[]{Livio Liechti}
\address{Peter Feller, Boston College, Department of Mathematics, Maloney Hall, Chestnut Hill, MA 02467, United States}
\email{peter.feller.2@bc.edu}
\thanks{The author is supported by the Swiss National Science Foundation (project no.\ 137548).}
\urladdr{}
\pagestyle{plain} 

%%%%%%%%%%%%%%%%%%%%%%%%%%%%%%%%%%%%%%%%%%%%%%%

\begin{abstract} {By a theorem of A'Campo, the eigenvalues of certain Coxeter transformations are positive real or lie on the unit circle.
By optimally bounding the signature of tree-like positive Hopf plumbings from below by the genus, we prove that at least two thirds of them lie on the unit circle.
In contrast, we show that for divide links, the signature cannot be linearly bounded from below by the genus.}
\end{abstract}

\maketitle

%%%%%%%%%%%%%%%%%%%%%%%%%%%%%%%%%%%%%%%%%%%%%%%

\section{Introduction}
\subsection{Tree-like positive Hopf plumbings and Coxeter systems}
Let $\Gamma$ be a finite tree embedded in the plane. The \textit{tree-like positive Hopf plumbing corresponding to $\Gamma$} is obtained by taking positive Hopf bands $H_i$ with core curves $\alpha_i$ that 
are in one-to-one correspondence with the vertices $v_i$ of $\Gamma$ and, starting from the root of $\Gamma$, plumbing them together such that $\alpha_i$ and $\alpha_j$ intersect each other exactly once if the vertices $v_i$ and $v_j$ are connected by an edge of $\Gamma$. Otherwise, the $\alpha_i$ do not intersect. 
The planar graph structure of $\Gamma$ provides a cyclic order on edges adjacent to a given vertex, which has to be preserved by the intersection points of the $\alpha_i$.
Here, plumbing denotes the operation of glueing two surfaces separated by a sphere together along some square on the sphere, as defined by Stallings~\cite{Sta}.
Again by Stallings, this procedure yields a fiber surface whose monodromy is conjugate to the product $T_{\alpha_1} \cdots T_{\alpha_n}$ of right Dehn twists along the $\alpha_i$.
Starting with the one edge graph with two vertices, this procedure yields the \textit{positive trefoil fiber}, the fiber surface of the left-handed trefoil knot, see Fig.~\ref{trefoilfiber}.

\begin{figure}
\begin{center}
\includegraphics [scale = 0.35] {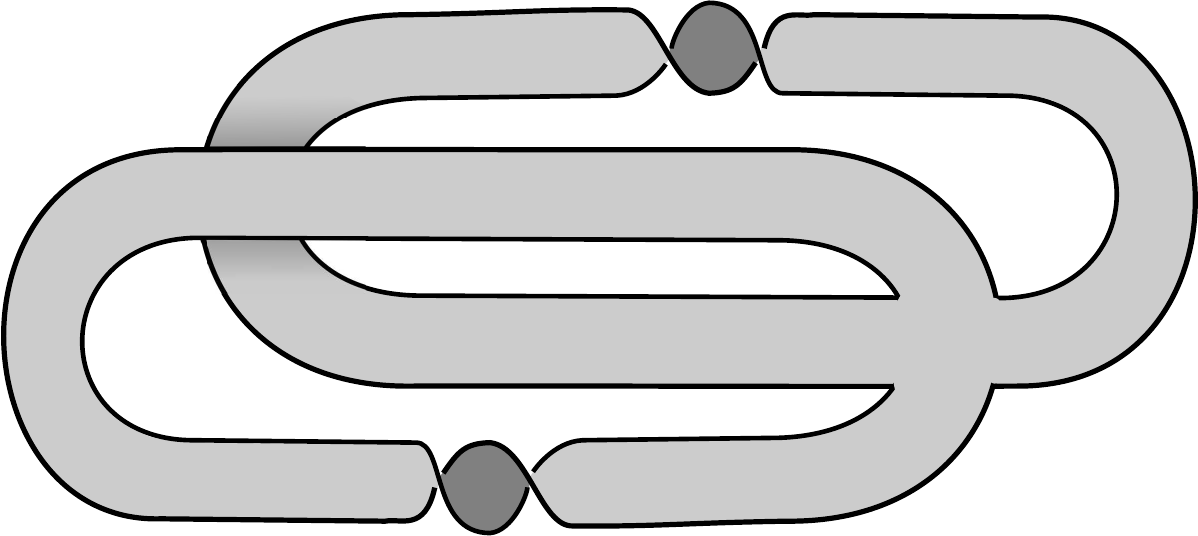}
\end{center}
\caption{}
\label{trefoilfiber}
\end{figure}

Let $\Gamma$ be a finite forest. The \textit{Coxeter system $(W,S)$ corresponding to $\Gamma$} is the group $W$ with generating set $S=\{s_1,...,s_n\}$, 
where the $s_i$ are in one-to-one correspondence with the vertices $v_i$ of $\Gamma$, relations $s_i^2=1$ for all $i$, the relation $(s_is_j)^3=1$ for all the $v_i$ and $v_j$ that are connected by an edge of $\Gamma$ and the relation $(s_is_j)^2=1$ for all the $v_i$ and $v_j$ that are not connected by an edge of $\Gamma$~\cite{Co1}. Note that except for $s_i^2=1$, these are the relations in the mapping class group of two positive Dehn twists along curves that intersect exactly once or that do not intersect, respectively. 
Let $V_{\Gamma}$ be the real vectorspace generated by the generators $s_i$ of $W$, equipped with the symmetric bilinear form $q_{\Gamma}$ given by $q_{\Gamma}(s_i,s_i) = -2$ and $q_{\Gamma}(s_i,s_j) = 1$ if and only if $v_i$ and $v_j$ are connected by an edge of $\Gamma$.
To every generator $s_i$ we associate the reflexion $R_i$ on the hyperplane orthogonal to $s_i$, given by $R_i(s_j) = s_j + q_{\Gamma}(s_i,s_j)s_i$. 
The \textit{Coxeter transformation} corresponding to $\Gamma$ is the product $R_1 \cdots R_n$ of all these reflections~\cite{Co2} and does, up to conjugation, not depend on the order of multiplication~\cite{Ste}. 
\begin{theorem}[{\cite{AC2}}] 
All eigenvalues of the Coxeter transformation corresponding to a finite forest are either positive real or lie on the unit circle.
\end{theorem}
The constructions of the monodromy of the tree-like positive Hopf plumbing and the Coxeter transformation corresponding to $\Gamma$ seem very similar. Indeed, A'Campo showed that for finite trees $\Gamma$, if one identifies the first homology of the positive Hopf plumbing corresponding to $\Gamma$ with the vector space $V_{\Gamma}$, the homological action of the monodromy $T_{\alpha_1} \cdots T_{\alpha_n}$ becomes conjugate to $- R_1 \cdots R_n$~\cite{AC5}.

\subsection{Signature}
For $p$ and $q$ coprime, one can show that the torus knot $T(p,q)$ has signature 
at least half the first Betti number of its fiber surface with the help of the recursive formulas proven by Gordon, Litherland and Murasugi~\cite{G-L-M}.
Furthermore, by Shinohara's cabling relation, this can be extended to all algebraic knots~\cite{Sh}. 
More recently, a linear lower bound that holds for all positive braids was given by Feller~\cite{Fe}.
We provide such a linear lower bound for the signature of tree-like positive Hopf plumbings. \newpage

\begin{thm}
\label{treesignature}
The signature of any tree-like positive Hopf plumbing is at least two thirds of the first Betti number.
\end{thm}

This result is optimal. Indeed, we construct tree-like positive Hopf plumbings of arbitrarily high genus but signature equal to exactly two thirds of the first Betti number.

It is well-known that for fiberd links, the Alexander polynomial equals the characteristic polynomial of the homological action of the monodromy. Thus, for a finite tree $\Gamma$, the Coxeter transformation corresponding to $\Gamma$ has an eigenvalue $\lambda$ if and only if $-\lambda$ is a zero of the Alexander polynomial of the tree-like positive Hopf plumbing corresponding to $\Gamma$. 
Furthermore, the absolute value of the signature of a link is a lower bound on the number of zeroes of the Alexander polynomial that lie on the unit circle. 
If the Alexander polynomial has only simple zeroes on the unit circle, this follows from a result of Stoimenow~\cite{Sto1}. 
In the Appendix, we give a general algebraic proof of this fact. Thus, we obtain the following corollary of Theorem~\ref{treesignature}, which applies exactly to the setting of A'Campo~\cite{AC2}.

\begin{cor}
\label{circulareigenvalues}
At least two thirds of the eigenvalues of the Coxeter transformation corresponding to a finite forest lie on the unit circle.
\end{cor}

Since the signature is a lower bound for the topological four-ball first Betti number by a result of Kauffman and Taylor~\cite{K-T}, we obtain yet another result as a corollary of Theorem~\ref{treesignature}.

\begin{cor}
\label{four-ball}
The topological four-ball first Betti number of any tree-like positive Hopf plumbing is at least two thirds of the ordinary first Betti number. 
\end{cor}

\subsection{Divides} A \textit{divide} is a finite collection $P$ of some generically immersed intervals or circles in the closed unit disc $D$. There is a canonical way of lifting these 
intervals and circles to a link $L(P) \subset S^3$, where $S^3$ lies inside the tangent bundle $TD$, which is identified with $D\times\R^2$. 
Divides and their associated links were introduced as a generalisation of algebraic links by A'Campo~\cite{AC3}. 
Furthermore, divide links of connected divides were shown to be fiberd by A'Campo~\cite{AC4} and, more precisely, to be plumbings of positive Hopf bands by Ishikawa~\cite{Is}. 
While the signature of any nontrivial divide knot is also strictly positive, we construct divide knots of arbitrarily high genus but signature equal to two. 
This is in strong contrast with the above examples where the signature is known to be linearly bounded from below by the genus.

\begin{figure}[h]
\begin{center}
\includegraphics [scale = 0.75] {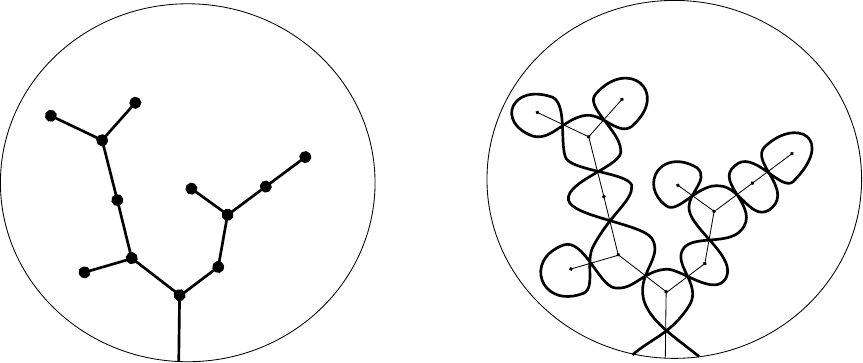}
\end{center}
\caption{}
\label{slalomdefinition}
\end{figure}
In \cite{AC5}, A'Campo introduced the \textit{slalom knots}, a certain class of knots which are both divides and tree-like positive Hopf plumbings. As divides, they are obtained in the following way: Take a rooted tree $\Gamma$ inside the unit disc~$D$ with the root on the boundary $\partial D$. Now, immerse an interval by the kind of slalom motion around the vertices of $\Gamma$ depicted in Fig.~\ref{slalomdefinition}. Equivalently, take $\Gamma$, insert a new vertex for every edge, then remove the root and its adjacent edge and do the tree-like positive Hopf plumbing that corresponds to this new planar graph. 
Note that different planar embeddings of the underlying abstract graph of $\Gamma$ yield different slalom knots which are related by mutation~\cite{Ge}. 
For this class of knots, a stronger version of Theorem~\ref{treesignature} holds and thus also a stronger version of the Corollaries~\ref{circulareigenvalues} and~\ref{four-ball}.

\begin{thm}
\label{slalomsignature}
 The signature of any slalom knot is at least three quarters of the first Betti number.
\end{thm}

This lower bound is optimal in the same sense as Theorem~\ref{treesignature}. Since the proofs of Theorem~\ref{slalomsignature} and its optimality use exactly the same ideas as the proofs for the corresponding statements for general tree-like positive Hopf plumbings, we omit them.

The last two sections are more open in nature. We ask whether divide knots are plumbings of positive trefoil fibers and what can be said about homological monodromies that are, up to a sign, conjugate to some Coxeter transformation. 
We furthermore conjecture that any zero of the Alexander polynomial of a positive braid link has real part smaller or equal to 1.
\newline

{\sc Acknowledgements.} I warmly thank Sebastian Baader for introducing me to the concepts and questions involved in this paper. I also thank Pierre Dehornoy and Peter Feller for valuable ideas and explanations, Luca Studer for his contribution to the calculations leading to Proposition~\ref{schnecken} and Filip Misev for pointing out the example given in Remark~C of the appendix. 
Finally, I thank the referee for helpful suggestions and corrections.

%%%%%%%%%%%%%%%%%%%%%%%%%%%%%%%%%%%%%%%%%%%%%%%%%%%%%%%%%%%%%%%%

\section{Signature of tree-like Hopf plumbings}
Let $\Gamma$ be a finite tree embedded in the plane. A matrix $S$ of a Seifert form of the corresponding positive Hopf plumbing with core curves $\alpha_i$ can easily be calculated. One obtains $S_{ii} =2$ for all $i$ and $S_{ij} = 1$ if and only if $\alpha_i$ and $\alpha_j$ intersect, otherwise $S_{ij} = 0$.
In order to show that for any $\Gamma$, the signature of this matrix is at least two thirds of its dimension,
we use Lemma~\ref{zerlegungslemma}, which, roughly speaking, gives a way of decomposing any tree 
into pieces on which the Seifert form is positive definite. We always identify the planar tree $\Gamma$ with its associated positive Hopf plumbing. When we write $\sigma(\Gamma)$ or $b_1(\Gamma)$, 
we mean the signature or the first Betti number of the associated Hopf plumbing. 
Actually, $b_1(\Gamma)$ is equal to the number of vertices of $\Gamma$.   

\begin{lem}
\label{zerlegungslemma}
Any tree $\Gamma$ with at least six vertices has a subtree $\Gamma_0 \subset \Gamma$ with at least six vertices such that $\sigma(\Gamma) \ge \sigma(\Gamma - \Gamma_0) + b_1(\Gamma_0) -2$.
\end{lem}

\begin{proof}[Proof of Theorem~\ref{treesignature}]
Let $\Gamma$ be a finite tree. Apply Lemma~\ref{zerlegungslemma} first to $\Gamma$, then to some tree of the forest $\Gamma - \Gamma_0$, etc. 
Apply Lemma~\ref{zerlegungslemma} as often as possible, say $r$ times, until the remaining forest does not have a tree with six or more vertices. 
Let $\Gamma_{0,i}$ be the subtree we obtain by the i-th use of Lemma~\ref{zerlegungslemma} and define the forest $\Gamma_{1,i} = \Gamma_{1,i-1} - \Gamma_{0,i}$ recursively, where $\Gamma_{0,1} = \Gamma_0$ and $\Gamma_{1,0} = \Gamma$. 
By Lemma~\ref{zerlegungslemma}, we get
$$\sigma(\Gamma) \ge \sigma(\Gamma_{1,1}) + (b_1(\Gamma_{0,1})-2) \ge ... \ge \sigma(\Gamma_{1,r}) + \sum_{i=1}^r (b_1(\Gamma_{0,i})-2).$$
It is easily checked that for any tree $\Gamma$ with at most five vertices, either $\sigma(\Gamma) = b_1(\Gamma)$ or $\sigma(\Gamma) = 4$. Since $\Gamma_{1,r}$ is a forest consisting only of trees with at most five vertices, we get that $\sigma(\Gamma_{1,r}) \ge \frac{4}{5}b_1(\Gamma_{1,r}) \ge \frac{2}{3}b_1(\Gamma_{1,r})$. 
Furthermore, since $b_1(\Gamma_{0,i}) \ge 6$, we have that $b_1(\Gamma_{0,i})-2 \ge \frac{2}{3}b_1(\Gamma_{0,i})$. This yields
$$\sigma(\Gamma_{1,r}) + \sum_{i=1}^r (b_1(\Gamma_{0,i})-2) \ge \frac{2}{3}(b_1(\Gamma_{1,r}) + \sum_{i=1}^r b_1(\Gamma_{0,i}))=\frac{2}{3}b_1(\Gamma).$$ 
Piecing all the inequalities together, we get that the signature $\sigma(\Gamma)$ is at least two thirds of the first Betti number $b_1(\Gamma)$, as desired.
\end{proof}

\begin{proof}[Proof of Lemma~\ref{zerlegungslemma}]
Let $\Gamma$ be a tree with at least six vertices. We choose a root for $\Gamma$ and orient all the edges away from the root. Let $v$ be a vertex that is outermost among the vertices of degree at least three. Every edge pointing away from $v$ defines a subtree of $\Gamma$ with only vertices of degree at most two: the maximal subtree containing the edge and $v$ but no other edge adjacent to $v$. Let $n = \text{deg}(v) - 1$ denote the number of such subtrees. Furthermore, let $k$ be the number of vertices outside (further away from the root) of $v$, let $v'$ be the vertex which is adjacent to $v$ but closer to the root and define $v''$ and $v'''$ analogously to $v'$, see Fig.~\ref{zerlegungsbeispiel}.

\begin{figure}[h]
  \def\svgwidth{170pt}
  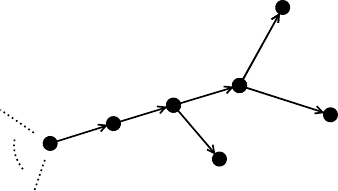
  \caption{}
  \label{zerlegungsbeispiel}
\end{figure}

\textit{Case 1: $k \ge 5$.} Let $\Gamma_0$ be the union of the $n$ subtrees specified above. Since on $\Gamma_0 - v$, the Seifert form is positive definite, the statement holds.

\textit{Case 2: $k = 4, n \le 3$.} Let $\Gamma_0$ be as in Case~1, but add the vertex $v'$ and the corresponding edge. Since on $\Gamma_0 - v'$, the Seifert form is positive definite, the statement holds. Note that in this case, $\Gamma - \Gamma_0$ need not be connected.

\textit{Case 3: $k = n = 4$.} Let $\Gamma_0$ be as in Case~2. Since the Seifert form is not positive definite on $\Gamma_0 - v'$, we cannot proceed as in Case~2.
The Seifert form of $\Gamma$ is given by the matrix

\begin{gather*}
  \begin{blockarray}{crrrrrrrrc}
    \begin{block}{(crr|rrrrrrc)}
      \ddots & \vdots & \vdots & \vdots & \vdots & \vdots & \vdots & \vdots & \vdots\\
      \cdots & \ast & \ast & \ast & 0 & 0 & 0 & 0 & 0 &\\ 
      \cdots & \ast & 2 & \ast & 0 & 0 & 0 & 0 & 0\\  \BAhhline{-------------~}
      \cdots & \ast & \ast & 2 & 1 & 0 & 0 & 0 & 0\\
      \cdots & 0 & 0 & 1 & 2 & 1 & 1 & 1 & 1\\
      \cdots & 0 & 0 & 0 & 1 & 2 & 0 & 0 & 0\\
      \cdots & 0 & 0 & 0 & 1 & 0 & 2 & 0 & 0\\
      \cdots & 0 & 0 & 0 & 1 & 0 & 0 & 2 & 0\\
      \cdots & 0 & 0 & 0 & 1 & 0 & 0 & 0 & 2\\
    \end{block}
  \end{blockarray}
\sim
  \begin{blockarray}{crrrrrrrrc}
    \begin{block}{(crr|rrrrrrc)}
      \ddots & \vdots & \vdots & \vdots & \vdots & \vdots & \vdots & \vdots & \vdots\\
      \cdots & \ast & \ast & \ast & 0 & 0 & 0 & 0 & 0 &\\ 
      \cdots & \ast & 2 & \ast & 0 & 0 & 0 & 0 & 0\\  \BAhhline{-------------~}
      \cdots & \ast & \ast & 2 & 1 & 0 & 0 & 0 & 0\\
      \cdots & 0 & 0 & 1 & 0 & 0 & 0 & 0 & 0\\
      \cdots & 0 & 0 & 0 & 0 & 2 & 0 & 0 & 0\\
      \cdots & 0 & 0 & 0 & 0 & 0 & 2 & 0 & 0\\
      \cdots & 0 & 0 & 0 & 0 & 0 & 0 & 2 & 0\\
      \cdots & 0 & 0 & 0 & 0 & 0 & 0 & 0 & 2\\
    \end{block}
  \end{blockarray},
\end{gather*}
where the bottom-right block corresponds to the restriction of the Seifert form to $\Gamma_0$, the top-left block to the restriction of the the Seifert form to $\Gamma - \Gamma_0$ and $\sim$ denotes a change of base.
By changing base again, we get that the Seifert form can be expressed by the matrix

\begin{gather*}
  \begin{blockarray}{crrrrrrrrc}
    \begin{block}{(crr|rrrrrrc)}
      \ddots & \vdots & \vdots & \vdots & \vdots & \vdots & \vdots & \vdots & \vdots\\
      \cdots & \ast & \ast & {0} & 0 & 0 & 0 & 0 & 0 &\\ 
      \cdots & \ast & 2 & {0} & 0 & 0 & 0 & 0 & 0\\  \BAhhline{-------------~}
      \cdots & {0} & {0} & 2 & 1 & 0 & 0 & 0 & 0\\
      \cdots & 0 & 0 & 1 & 0 & 0 & 0 & 0 & 0\\
      \cdots & 0 & 0 & 0 & 0 & 2 & 0 & 0 & 0\\
      \cdots & 0 & 0 & 0 & 0 & 0 & 2 & 0 & 0\\
      \cdots & 0 & 0 & 0 & 0 & 0 & 0 & 2 & 0\\
      \cdots & 0 & 0 & 0 & 0 & 0 & 0 & 0 & 2\\
    \end{block}
  \end{blockarray}
\sim
  \begin{blockarray}{crrrrrrrrc}
    \begin{block}{(crr|rrrrrrc)}
      \ddots & \vdots & \vdots & \vdots & \vdots & \vdots & \vdots & \vdots & \vdots\\
      \cdots & \ast & \ast & 0 & 0 & 0 & 0 & 0 & 0 &\\ 
      \cdots & \ast & 2 & 0 & 0 & 0 & 0 & 0 & 0\\  \BAhhline{-------------~}
      \cdots & 0 & 0 & 2 & 0 & 0 & 0 & 0 & 0\\
      \cdots & 0 & 0 & 0 & -\frac{1}{2} & 0 & 0 & 0 & 0\\
      \cdots & 0 & 0 & 0 & 0 & 2 & 0 & 0 & 0\\
      \cdots & 0 & 0 & 0 & 0 & 0 & 2 & 0 & 0\\
      \cdots & 0 & 0 & 0 & 0 & 0 & 0 & 2 & 0\\
      \cdots & 0 & 0 & 0 & 0 & 0 & 0 & 0 & 2\\
    \end{block}
  \end{blockarray}.
\end{gather*}
Since the changes of base we applied never changed the top-left block, we get that $\sigma(\Gamma) = \sigma(\Gamma - \Gamma_0) + 4$. 

\textit{Case 4: $k=3$, $n =2$, $\text{\emph{deg}}(v')=2$.} Let $\Gamma_0$ be as in Case~2 but add the vertex $v''$ and the corresponding edge. Since on $\Gamma_0 - v''$, the Seifert form is positive definite, the statement holds. Again, $\Gamma - \Gamma_0$ need not be connected.

\textit{Case 5: $k=n=3$, $\text{\emph{deg}}(v')=2$.} Let $\Gamma_0$ be as in Case~4. This works very similar to Case~3. Writing down a matrix for the Seifert form of $\Gamma$ with the Seifert form restricted to $\Gamma_0$ in the bottom-right block and applying a change of base, we get that $\sigma(\Gamma) = \sigma(\Gamma - \Gamma_0) + 4$.

\textit{Case 6: $k=n=2$, $\text{\emph{deg}}(v')=\text{\emph{deg}}(v'')=2$.} Let $\Gamma_0$ be as in Case~4 but add the vertex $v'''$ and the corresponding edge. Since on $\Gamma_0 - v'''$, the Seifert form is positive definite, the statement holds. Again, $\Gamma - \Gamma_0$ need not be connected.

\textit{Case 7: none of the other cases apply}. If three or four vertices lie outside of $v'$, then let $\Gamma_0$ be as in Case~4. Since Case~6 does not apply, at least five vertices lie outside of $v''$. Since none of the other cases apply, it is easily checked that on $\Gamma_0 - v''$, the Seifert form is positive definite and the statement holds. If at least five vertices lie outside of $v'$, then let $\Gamma_0$ be as in Case~2. Again none of the other cases apply, the Seifert form is positive definite on $\Gamma_0 - v'$ and the statement holds. Once more, $\Gamma - \Gamma_0$ need not be connected.
\end{proof}

\begin{remark}
\label{optimality}
 The optimality of Theorem~\ref{treesignature} follows directly from Case~5 in the proof of Lemma~\ref{zerlegungslemma}. The signature of the link
 corresponding to the tree dealt with in this case is~4, while its first Betti number is~6. By the reasoning in the proof, 
 glueing such a tree to another tree always adds~4 to the signature and~6 to the first Betti number. 
 Like this, one always obtains a tree with signature equal to exactly two thirds of the first Betti number. This constructions yields
 links of arbitrarily high genus. Interestingly, one can show that these examples have topological four-ball first Betti number equal to the signature.
 This leads to the question whether this holds for any tree-like positive Hopf plumbing. This and similar questions will be subject to future research.

\begin{question}
 Is the topological four-ball first Betti number of any tree-like positive Hopf plumbing equal to the signature?
\end{question}
\end{remark}

%%%%%%%%%%%%%%%%%%%%%%%%%%%%%%%%%%%%%%%%%%%%%%%%%%%%%%%%%%%%%%%%

\section{Divides} 
A matrix $S$ of a Seifert form of a given divide link can be calculated as described in~\cite{B-D1}. As a basis of the first homology of the fiber surface, take the core curves $\alpha_i$ of the positive Hopf bands used in Ishikawa's plumbing constructions~\cite{Is}. Thus, the basis consists of one curve for each inner face and one curve for each double point of the divide. 
Drawing pictures of the various situations, one obtains $S_{ii} = 2$ for all $i$ and $S_{ij} = n$ if $\alpha_i$ and $\alpha_j$ are \textit{n}-fold adjacent,
where \textit{n-fold adjacency} is defined as follows. Two curves corresponding to inner faces are called \textit{n-fold adjacent}, if the inner faces have $n$ common edges. A curve corresponding to a double point and a curve corresponding to an inner face are called \textit{n-fold adjacent} if the double point occurs $n$ times in the boundary of the inner face. Two curves corresponding to different double points are \textit{0-fold adjacent}. 

It is a rather simple result that the signature of any nontrivial divide knot is strictly positive and we only give a sketch of the proof. Since two curves corresponding to different double points are 0-fold adjacent, the Seifert form is positive definite on the subspace of the first homology spanned by the curves corresponding to all the double points. An Euler characteristic argument shows that these account for exactly half the genus. This shows that the signature is greater or equal to zero. To show strict positivity, one again uses an Euler characteristic argument to find an inner face $F$ with at most three double points on its boundary. Since the matrices
$$\begin{pmatrix}
 2 & 1 & 1 & 1\\
 1 & 2 & 0 & 0\\
 1 & 0 & 2 & 0\\
 1 & 0 & 0 & 2
\end{pmatrix},
\begin{pmatrix}
 2 & 1 & 1 \\
 1 & 2 & 0 \\
 1 & 0 & 2 
\end{pmatrix},
\begin{pmatrix}
 2 & 1 \\
 1 & 2 
\end{pmatrix} \text{and}
\begin{pmatrix}
 2 
\end{pmatrix}
$$
are positive definite, the Seifert form is still positive definite on the subspace spanned by the curves corresponding to all the double points and the curve corresponding to $F$.
Thus, the signature of any nontrivial divide knot is strictly positive. In fact, one can easily adapt this proof to show that the signature of any nontrivial divide link is  also strictly positive.

However, we focus on the fact that this result is optimal. For any number~$n\ge 1$, Proposition~\ref{schnecken} provides an example of a divide with $n$ double points such that the 
signature of the corresponding knot is equal to two. Since the number of double points of a divide is equal to the genus of the associated divide knot, this
shows that the signature of divide knots cannot be linearly bounded from below by the genus. 

\begin{figure}
\hfill
\includegraphics [scale = 0.25] {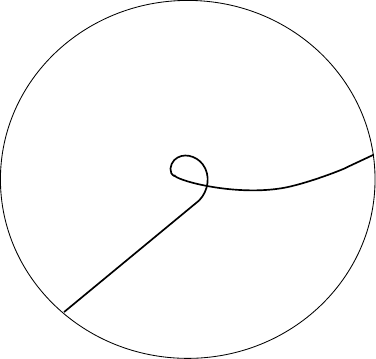} \hfill
\includegraphics [scale = 0.25] {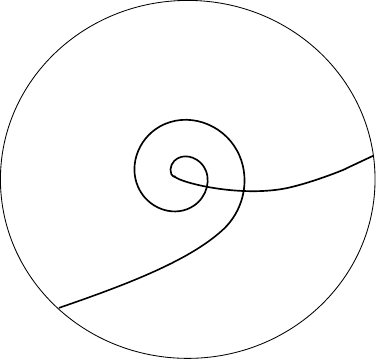} \hfill 
\includegraphics [scale = 0.25] {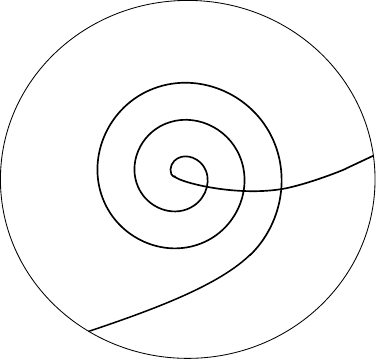} \hfill
\includegraphics [scale = 0.25] {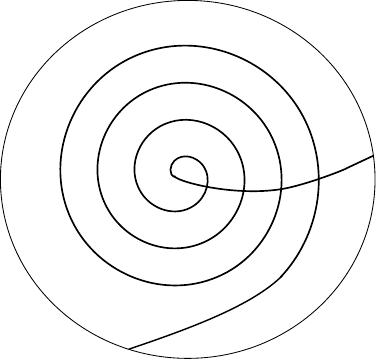} \hfill
etc.
\caption{}
\label{signature2}
\end{figure}

\begin{prop}
\label{schnecken}
The signature of any divide knot of the family depicted in Fig.~\ref{signature2} is equal to two.
\end{prop} 

\begin{proof}
Choose the following basis of the first homology of the fiber surface. First take the curves corresponding to the double points from inside out, then take the curves corresponding to the inner faces from inside out. For this choice of basis, the Seifert form of the divide with $n$ crossings is given by the $2n\times 2n$ matrix
$$
S_{2n} =
\begin{pmatrix}
 A_n & B_n\\
 B_n^t & D_n
\end{pmatrix},
$$ 
where 
$$A_n = 
\begin{pmatrix}
 2  \\
  & \ddots \\
  & & 2
\end{pmatrix}, 
B_n = 
\begin{pmatrix}
1& 2 & 1 \\
& \ddots & \ddots & \ddots \\
&  & 1 & 2 & 1 \\
&  & & 1 & 2 \\
&  & & & 1
\end{pmatrix},$$ $$
D_n = 
\begin{pmatrix}
2 & 1 &  \\[0.4em]
1 & 2 & 2 \\
   & 2 & 2 & \ddots  \\
& & \ddots & \ddots & 2 \\[0.4em]
& & & 2 & 2
\end{pmatrix}.
$$
Since $A_n$ is invertible, the formula
\begin{align*} S_{2n} = 
\begin{pmatrix}
 A_n & B_n\\
 B_n^t & D_n
\end{pmatrix} =
\begin{pmatrix}
 A_n & 0\\
 B_n^t & Id
\end{pmatrix}
\begin{pmatrix}
 Id & A_n^{-1}B_n\\
 0& D_n - B_n^tA_n^{-1}B_n
\end{pmatrix}
\end{align*}
holds. Furthermore, since $A_n$ is a scalar matrix, it certainly commutes with $B_n$ and we obtain
\begin{align*}\det(S_{2n}) &= \det(A_n)\det(D_n - B_n^tA_n^{-1}B_n) 
\\ &= (-1)^n\det(B_n^tB_n - A_nD_n).\end{align*}
Calculating $B_n^tB_n - A_nD_n$ yields the matrix

$$
\begin{pmatrix}
-3 & 0 & 1 \\
0 & 1 & 0 & \ddots \\
1 & 0 & 2 & \ddots &  1 \\
 & \ddots & \ddots & \ddots & 0\\
 &  & 1 & 0 & 2
\end{pmatrix},
$$ \newline
which is easily brought into upper triangular form. This upper triangular form then shows that the determinant of $B_n^tB_n - A_nD_n$ is always negative. Thus, we have that $\text{sign}(\det(S_{2n})) = (-1)^n(-1) = (-1)^{n+1}$. Now we conclude the proof by induction over $n$. A quick calculation shows that $S_4$ has signature two. Now assume that the signature of $S_{2n}$ is equal to two. The determinant of $S_{2(n+1)}$ is of opposite sign than the determinant of $S_{2n}$. Additionally, $S_{2(n+1)}$ contains $S_{2n}$ as a minor. Thus, in addition to the eigenvalues of $S_{2n}$, $S_{2(n+1)}$ has exactly one positive and one negative eigenvalue. In particular, the signatures of $S_{2(n+1)}$ and $S_{2n}$ are equal. 
\end{proof}

%%%%%%%%%%%%%%%%%%%%%%%%%%%%%%%%%%%%%%%%%%%%%%%%%%%%%%%%%%%

\section{Hopf vs. trefoil plumbing}
It was shown by Giroux and Goodman that the fiber surface of any fiberd link is obtained from the disc by consecutively plumbing and deplumbing some Hopf bands~\cite{G-G}. 
They include a remark suggesting that for fiberd knots, their plumbing and deplumbing operations can be made two by two, such that the intermediate steps are always fiberd knots. 
Thus, the fiber surface of any fiberd knot would be a plumbing and deplumbing of trefoil and figure eight fibers. We give an example that shows that the deplumbing operation is necessary even in the case where the fiber surface is actually a plumbing of Hopf bands.  

Since a plumbing of two surfaces is quasipositive if and only if the two surfaces are quasipositive~\cite{Ru}, any plumbing of positive Hopf bands is quasipositive.
Furthermore, the only way of obtaining such a surface as a plumbing of trefoil and figure eight fibers is as a plumbing of positive trefoil fibers.
In Example~\ref{keintrefoilplumbingbeispiel}, we describe a fiberd knot whose fiber surface is a plumbing of four positive Hopf bands but not a plumbing of two positive trefoil fibers. 
\begin{lem}
\label{trefoilplumbingsignature}
Plumbing a Hopf band changes the signature by at most one
and plumbing a positive trefoil fiber never reduces the signature.
\end{lem}

\begin{proof}
Note that by choosing bases correctly, the matrix of the Seifert form before the plumbing is a minor of the matrix of the Seifert form after the plumbing. 
Since plumbing a Hopf band changes the first Betti number by one, the first statement follows. 
Similarly, the second statement follows from the fact that only one of the core curves of the two positive Hopf bands forming the positive trefoil fiber touches the surface the positive trefoil fiber is plumbed onto. 
\end{proof}

\begin{figure}
\begin{center}
\includegraphics [scale = 0.47] {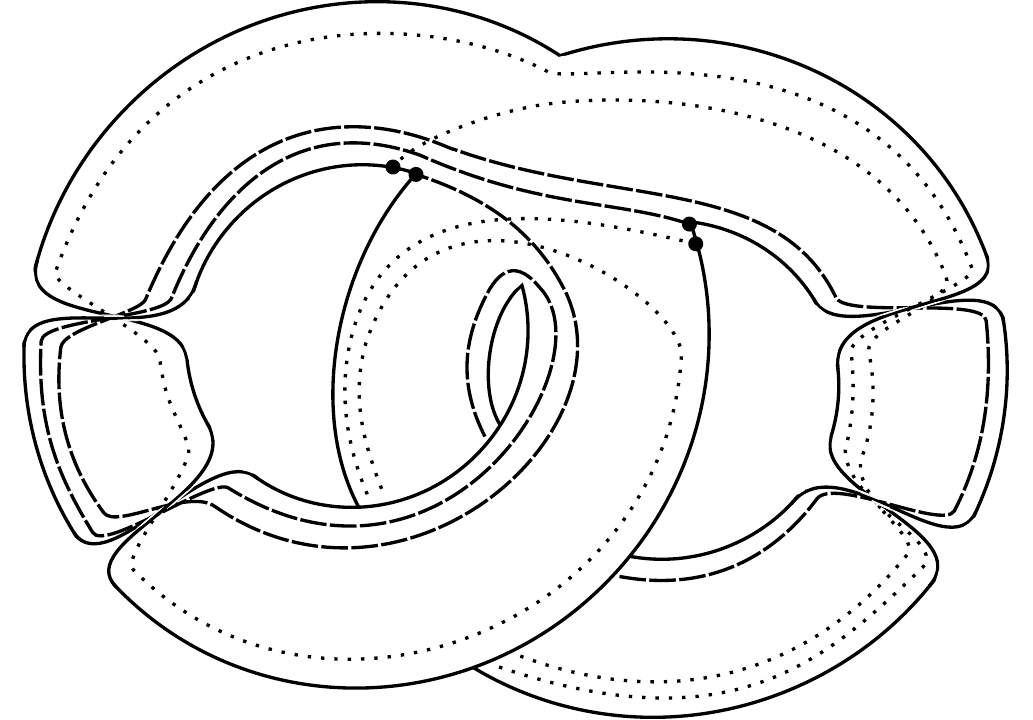}
\end{center}
\caption{}
\label{keintrefoilplumbing}
\end{figure}

\begin{example}
\label{keintrefoilplumbingbeispiel}
Fig.~\ref{keintrefoilplumbing} shows the fiber surface of the left-handed trefoil knot and two embedded intervals with endpoints on the boundary of the surface. 
Every such embedded interval describes a Hopf plumbing. Now plumb a positive Hopf band first along the dashed interval and then another one along the dotted interval. 
By choosing the suitably oriented core curves of the plumbed Hopf bands as a basis of the first homology, the entries of the matrix for the Seifert form become just the intersection numbers of the core curves. 
Calculating these yields the matrix  
$$
\begin{pmatrix}
 2 & 1 & 3 & 2\\
 1 & 2 & 2 & 3\\
 3 & 2 & 2 & 4\\
 2 & 3 & 4 & 2
\end{pmatrix},
$$
which has signature equal to zero. Since a plumbing of positive trefoil fibers has strictly positive signature by Lemma~\ref{trefoilplumbingsignature}, 
the surface obtained after plumbing positive Hopf bands along the dashed and the dotted interval is not a plumbing of positive trefoil fibers. 
\end{example}

For specific classes of fiberd knots, however, the deplumbing operation need not be necessary. 
For example, it was shown by Baader and Dehornoy that the fiber surface of any positive braid knot is a plumbing of positive trefoil fibers~\cite{B-D2}. Another example of positive trefoil plumbings are slalom knots. In fact, Lewark showed that a tree-like positive Hopf plumbing has one boundary component if and only if it is actually a plumbing of positive trefoil fibers~\cite{Lew}.

One can show that the examples provided by Proposition~\ref{schnecken} are plumbings of positive trefoil fibers. By Lemma~\ref{trefoilplumbingsignature}, these examples actually minimise 
the signature among plumbings of positive trefoil fibers. 

\begin{question}
 \label{Qtrefoilplumbing}
 Is the fiber surface of any divide knot a plumbing of positive trefoil fibers?
\end{question}

\section{Coxeter systems and the location of zeroes of the Alexander polynomial}
In 2002, Hoste conjectured the following result on the location of zeroes of the Alexander polynomial of alternating knots. 
\begin{conjecture}[J. Hoste, 2002]
The real part of any zero of the Alexander polynomial of an alternating knot is strictly greater than -1.
\end{conjecture}

In the restricted case of two-bridge knots, a first lower bound was proven by Lyubich and Murasugi~\cite{L-M}. 
Recently, this bound has been improved by Koseleff and Pecker~\cite{K-P} and independently by Stoimenow \cite{Sto2}.
Furthermore, Hirasawa and Murasugi constructed many examples of alternating links with all zeroes of the Alexander polynomial real and strictly positive~\cite{H-M}.
Interestingly, we get a very similar but antipodal result for tree-like positive Hopf plumbings. Since all the zeroes of the Alexander polynomial are either negative real or on the unit circle, we get that 
the real part of any zero of the Alexander polynomial is smaller or equal to 1 with strict inequality if and only if the link in question is actually a knot. 

Actually, the homological monodromy of any plumbing of positive Hopf bands whose core curves intersect at most once is, up to a sign, conjugate to some Coxeter transformation. 
Since the corresponding Coxeter graph need not be simply connected, there are, in general, several conjugacy classes of Coxeter transformations. 
If the Coxeter graph is bipartite, there is still a distinguished Coxeter transformation, the \textit{bicolored Coxeter transformation}, 
for which the eigenvalues are either positive real or lie on the unit circle, see e.g.~\cite{Mc}. 
The homological monodromy, however, has no particular reason to be in the conjugacy class of this bicolored Coxeter transformation. It is indeed not difficult to construct examples of positive braids such that 
the corresponding Coxeter graph is bipartite but the homological monodromy of the positive braid link has non-real eigenvalues outside the unit circle. 

\begin{question}
 What can be said about homological monodromies that are, up to sign, conjugate to some Coxeter transformation?
\end{question}

The distribution of zeroes of the Alexander polynomial of positive braids still seems very particular. 
See for example Figure~25 appearing towards the end of~\cite{De}, which shows the distributions of zeroes of the Alexander polynomial of random positive braids. 
Considering this figure leads to the following conjecture, again antipodal to Hoste's conjecture.

\begin{conj} 
\label{nullstellen}
The real part of any zero of the Alexander polynomial of a positive braid link $L$ is smaller or equal to 1 with strict inequality if and only if $L$ is a knot.
\end{conj}

%%%%%%%%%%%%%%%%%%%%%%%%%%%%%%%%%%%%%%%%%%%%%%%%%%%%%%%%%%%%%%%%%%%%%%

\section*{{Appendix. Signature and the Alexander polynomial}}

\begin{center}\textsc{Peter Feller and Livio Liechti}\end{center}
$ $

In this appendix, we prove that the absolute value of the signature of a link is a lower bound for the number of zeroes of the Alexander polynomial that lie on the unit circle. 
We believe that this result is known to a certain extent. For example, if the Alexander polynomial has only simple zeroes on the unit circle, it follows from a result of Stoimenow~\cite{Sto1}. 
However, we do not know of any reference providing the general statement. 

\begin{thmA}
The Alexander polynomial $\Delta_L(t)$ of any link $L$ is either identically zero or has at least $\vert \sigma(L)\vert$ zeroes (counted with multiplicity) on the unit circle.
\end{thmA}

We start by recalling the definitions of the Alexander polynomial,
the signature and its generalisations, the $\omega$~--~signatures, as defined by Levine and Tristram~\cite{Lev,Tr}.
Let $L$ be a link and $A_{n\times n}$ be a Seifert matrix for $L$. The Alexander polynomial $\Delta_L$ of $L$ is defined, up to normalisation, as $$\Delta_L(t)=\text{det}(tA-A^T),$$ 
and for $\omega$ on the unit circle in $\C$, the $\omega$~--~signature $\sigma_{\omega}(L)$ of $L$ is defined as the signature of the hermitian matrix $$M_\omega = (1-\omega)A+(1-\bar{\omega})A^T.$$
For $\omega=-1$, this equals the definition of the classical signature invariant $\sigma(L)$, and for $\omega=1$, it is equal to zero.

As $\omega$ runs around the unit circle, the eigenvalues $\lambda_i(\omega)\in\R$ of $M_\omega$ depend continuously on $\omega$. It may happen that one of the eigenvalues $\lambda_i(\omega)$ passes
through zero for some $\omega_0$. In this case, $\text{det}M_{\omega_0} = 0$ and the $\omega$~--~signature $\sigma_{\omega}(L)$ may have a discontinuity at $\omega_0$. We say that the
$\omega$~--~signature \textit{jumps} at $\omega_0$ and call $$j_{\omega_0} = \frac{1}{2}(\sigma_{\omega_0 + \epsilon}(L) - \sigma_{\omega_0 - \epsilon}(L)) \in \Z$$ the \textit{signature jump at} $\omega_0$.

It is well-known that the signature can only jump at zeroes of the Alexander polynomial. Indeed, for any $\omega$ on the unit circle, it holds that 
$$M_\omega = -(1-\bar{\omega})(\omega A-A^T),$$ and consequently, for any discontinuity $\omega_0 \ne 1$,
$$\Delta_L(\omega_0) = \text{det}(\omega_0 A-A^T)) = \text{det}((-(1-\bar{\omega_0})^{-1}M_{\omega_0}) = 0.$$

\begin{lemA}
For any zero $\omega_0 \ne 0$ of the Alexander polynomial, the order is greater or equal to the nullity of $\omega_0 A -A^T$.
\end{lemA}

\begin{proof}

Consider the matrix $tA -A^T \in \text{Mat}_{n\times n}(\C[t])$.
There exist matrices $P,Q \in \text{GL}(\C[t])$ such that $P(tA -A^T)Q$ is in Smith normal form,
i.e.\ $P(tA -A^T)Q$ is a diagonal matrix with entries $\alpha_i \in \C[t]$ and such that $\alpha_i |\alpha_{i+1}$, see~\cite{Sm}. 
Setting $c=\text{det}(P)\text{det}(Q) \in \C$, we obtain
\begin{align*}c\cdot\Delta_K(t)&=\text{det}(P)\text{det}(tA -A^T)\text{det}(Q)\\
&=\text{det}(P(tA -A^T)Q)\\
&=\alpha_1 \cdots \alpha_n\\
&=(t-\omega_0)^m\cdot p(t),
\end{align*}
where $p(\omega_0) \ne 0$. The number of $\alpha_i$ that have a (perhaps multiple) zero at $\omega_0$ is exactly equal to the nullity of $\omega_0 A -A^T$. Therefore, we get that
$m$ is greater or equal to the nullity of $\omega_0 A -A^T$. However, $m$ is exactly the order of the zero of the Alexander polynomial at $\omega_0$. 
\end{proof}

\begin{remarkA}
{\em The order of the zero of the Alexander polynomial at $\omega_0$ can actually be strictly greater than the nullity of $\omega_0 A -A^T$. As an example, take $\omega_0 = -1$ for the link of the singularitiy at zero of the curve given by $(x^2 +y^3)(x^3+y^2)$, see~\cite{AC1}. The monodromy matrix given towards the end has an eigenvalue $\omega_0 = -1$ with algebraic multiplicity equal to two but geometric multiplicity equal to one. }
\end{remarkA}

Since for $\omega_0 \ne 1$, the jump $|j_{\omega_0}|$ is less or equal to the nullity of $M_{\omega_0}$ and the nullity of $M_{\omega_0}$ equals the nullity of $\omega_0 A -A^T$, 
we get the following proposition relating the signature jumps to the order of the zeroes of the Alexander polynomial as a consequence of Lemma~B.

\begin{propA}
If the $\omega$~--~signature $\sigma_{\omega}(K)$ jumps at $\omega_0 \ne 1$, then the signature jump $j_{\omega_0}$ at $\omega_0$ is smaller or equal to the order of the zero of the Alexander polynomial at $\omega_0$.
\end{propA}

\begin{proof}[Proof of Theorem~A]
So far we examined the case $\omega_0 \ne 1$. In order to make a statement about the total number of zeroes of the Alexander polynomial that lie on the unit circle, we also have to study the situation at $\omega_0 = 1$. If $\omega$ tends towards $1$, the eigenvalues $\lambda_i(\omega)$ of $M_\omega$ tend, up to some normalisation constant, to the eigenvalues of $iA-iA^T$. 
Since $A-A^T$ is skew-symmetric, the signature of $iA-iA^T$ is zero. Therefore, for $\omega$ close enough to $1$, 
the modulus $\vert \sigma_{\omega}(L)\vert = \vert \sigma(M_\omega)\vert$ is bounded from above by the nullity of $A-A^T$, which in turn is bounded from above by the order of the zero of the Alexander polynomial at 1 by Lemma~B. Together with Proposition~D, this yields the desired result.
\end{proof}

%%%%%%%%%%%%%%%%%%%%%%%%%%%%%%%%%%%%%%%%%%%%%%%%%%%%%%%%%%%%%%%%%%%%%%%%%%%%

\bibliographystyle{siam}

\end{document}

%% file: zerlegungsbeispiel.pdf_tex
%% Creator: Inkscape inkscape 0.48.2, www.inkscape.org
%% Accompanies image file 'zerlegungsbeispiel.pdf' (pdf, eps, ps)
\begingroup%
  \makeatletter%
  \providecommand\color[2][]{%
    \errmessage{(Inkscape) Color is used for the text in Inkscape, but the package 'color.sty' is not loaded}%
    \renewcommand\color[2][]{}%
  }%
  \providecommand\transparent[1]{%
    \errmessage{(Inkscape) Transparency is used (non-zero) for the text in Inkscape, but the package 'transparent.sty' is not loaded}%
    \renewcommand\transparent[1]{}%
  }%
  \providecommand\rotatebox[2]{#2}%
  \ifx\svgwidth\undefined%
    \setlength{\unitlength}{162.23100178bp}%
    \ifx\svgscale\undefined%
      \relax%
    \else%
      \setlength{\unitlength}{\unitlength * \real{\svgscale}}%
    \fi%
  \else%
    \setlength{\unitlength}{\svgwidth}%
  \fi%
  \global\let\svgwidth\undefined%
  \global\let\svgscale\undefined%
  \makeatother%
  \begin{picture}(1,0.56298458)%
    \put(0,0){\includegraphics[width=\unitlength]{zerlegungsbeispiel.pdf}}%
    \put(0.65317003,0.34358756){\color[rgb]{0,0,0}\makebox(0,0)[lb]{\smash{$v$}}}%
    \put(0.46928917,0.28290538){\color[rgb]{0,0,0}\makebox(0,0)[lb]{\smash{$v'$}}}%
    \put(0.29567289,0.23120591){\color[rgb]{0,0,0}\makebox(0,0)[lb]{\smash{$v''$}}}%
    \put(0.12527694,0.18418607){\color[rgb]{0,0,0}\makebox(0,0)[lb]{\smash{$v'''$}}}%
  \end{picture}%
\endgroup%